\pdfoutput=1
\documentclass{article}
\usepackage[english]{babel}
\usepackage[a4paper]{geometry}
\usepackage{amsmath,amssymb}
\usepackage{mathrsfs}
\usepackage{amsthm,thmtools}
\usepackage{enumitem}
\usepackage[rgb]{xcolor}
\usepackage{tikz}
\usepackage{lmodern}
\usepackage{microtype}
\usepackage{hyperref}

\author{Josse van Dobben de Bruyn}
\title{Representations and Semisimplicity of\texorpdfstring{\\}{} Ordered Topological Vector Spaces}
\date{24 September 2020}

\makeatletter
\newcommand{\myaddr}[1]{\gdef\my@address{\par\textsc{#1}}}
\newcommand{\mycuraddr}[1]{\gdef\my@curaddr{\par\textsc{#1}}}
\newcommand{\myemail}[1]{\gdef\my@email{\par\textit{E-mail address:} \texttt{\href{mailto:#1}{#1}.}}}
\AtEndDocument{\par\bigskip\my@address\my@curaddr\my@email}
\newcommand{\mysubjclass}[2][2010]{\gdef\my@subjclass{#1 \textit{Mathematics Subject Classification}. #2.}}
\newcommand{\mykeywords}[1]{\gdef\my@keywords{\textit{Key words and phrases}. #1.}}
\newcommand{\mymaketitle}{%
	\let\@oldthanks\@thanks%
	\gdef\@thanks{\@oldthanks\footnotetext{\my@subjclass}\footnotetext{\my@keywords}}%
	\maketitle
}
\makeatother

\myaddr{Mathematical Institute, Leiden University, P.O.~Box 9512, 2300 RA Leiden, The Netherlands}
\mycuraddr{Delft Institute of Applied Mathematics, Delft University of Technology, Van Mourik Broekmanweg 6, 2628 XE Delft, The Netherlands}
\myemail{J.vanDobbendeBruyn@tudelft.nl}
\mysubjclass[2020]{Primary: 46A40, Secondary: 46B40, 06F20}
\mykeywords{Convex cone, partially ordered topological vector space, functional representation}

\def\mylinkcolor{black!75!blue}
\makeatletter
\hypersetup{pdfauthor=\@author,pdftitle=\@title,colorlinks=true,linkcolor=\mylinkcolor,citecolor=\mylinkcolor,urlcolor=\mylinkcolor}
\makeatother

\newcommand{\myautoref}[2]{\hyperref[#2]{\autoref*{#1}\ref*{#2}}}  
\newcommand{\mysecref}[1]{\hyperref[#1]{\S\ref*{#1}}} 
\newcommand{\myref}[1]{\textit{\ref{#1}}}

\newcommand{\hair}{\ifmmode\mskip1mu\else\kern0.08em\fi}
\newcommand{\exampleqed}{\ensuremath{\scriptstyle\triangle}}

\numberwithin{equation}{section}

\declaretheorem[style=definition,sibling=equation]{definition}
\declaretheorem[style=definition,qed=\exampleqed,sibling=definition]{example}
\declaretheorem[style=definition,sibling=definition]{remark}

\declaretheorem[style=plain,sibling=definition]{theorem}
\declaretheorem[style=plain,numbered=unless unique]{theorem S}
\declaretheorem[style=plain,numbered=unless unique]{theorem K}
\declaretheorem[style=plain,sibling=definition]{lemma}
\declaretheorem[style=plain,sibling=definition]{proposition}
\declaretheorem[style=plain,sibling=definition]{corollary}

\declaretheorem[style=plain,title={Theorem}]{AlphTheorem}

\declaretheorem[style=plain,title={Lemma},sibling=AlphTheorem]{AlphLemma}
\declaretheorem[style=plain,title={Corollary},sibling=AlphTheorem]{AlphCorollary}

\newcommand{\N}{\ensuremath{\mathbb{N}}}
\newcommand{\R}{\ensuremath{\mathbb{R}}}

\DeclareSymbolFont{bbold}{U}{bbold}{m}{n}
\DeclareSymbolFontAlphabet{\mathbbold}{bbold}
\newcommand{\one}{\ensuremath{\mathbbold{1}}}

\newcommand{\weak}{{w}}
\newcommand{\weakstar}{{w*}}
\newcommand{\discr}{d}

\newcommand{\ball}[1]{B_{{#1}}}
\DeclareMathOperator{\cone}{cone}

\newcommand{\algdual}{^*}
\newcommand{\topdual}{'}
\newcommand{\topdualdual}{''}

\DeclareMathOperator{\lineal}{lin}
\newcommand{\mywedge}{\mathcal{K}}
\DeclareMathOperator{\spn}{span}

\newcommand{\lintop}{\mathfrak T}
\newcommand{\seminorm}{\rho}
\newcommand{\Ell}{\mathscr{L}}

\newcommand{\ABt}{t} 
\newcommand{\ECt}{t} 
\newcommand{\OEt}{t} 

\newcommand{\Cp}{C_p}
\newcommand{\Ck}{C_k}

\DeclareMathOperator{\mylinsum}{lin}
\DeclareMathOperator{\mylcsum}{lc}
\newcommand{\linsum}{\bigoplus^{\mylinsum}}
\newcommand{\lcsum}{\bigoplus^{\mylcsum}}

\DeclareMathOperator{\orad}{orad}
\DeclareMathOperator{\torad}{torad}

\begin{document}
\mymaketitle
\begin{abstract}
	This paper studies ways to represent an ordered topological vector space as a space of continuous functions,
	extending the classical representation theorems of Kadison and Schaefer.
	Particular emphasis is put on the class of \emph{semisimple} spaces, consisting of those ordered topological vector spaces that admit an injective positive representation to a space of continuous functions.
	We show that this class forms a natural topological analogue of the \emph{regularly ordered} spaces defined by Schaefer in the 1950s,
	and is characterized by a large number of equivalent geometric, algebraic, and topological properties.
\end{abstract}

\section{Introduction}
\subsection{Outline}
A common technique in functional analysis and other branches of mathematics is to study an abstract space $E$ by studying its \emph{representations}, special types of maps from $E$ to concrete spaces of continuous functions or operators.
For instance, a commutative Banach algebra admits a norm-decreasing homomorphism to a space of continuous functions (the \emph{Gelfand representation}), and a $C^*$-algebra admits an injective $*$-homomorphism to an algebra of operators on a Hilbert space (the \emph{GNS representation}).
For ordered vector spaces, the most famous result in this direction is Kadison's representation theorem:

\begin{theorem K}[{Kadison \cite[Lemma 2.5]{Kadison-representation}}]
	\label{thm:Kadison-representation}
	Let $E$ be a \textup(real\textup) ordered vector space whose positive cone $E_+$ is Archimedean and contains an order unit $e$. If $E$ is equipped with the corresponding order unit norm $\lVert \:\cdot\: \rVert_e$, then $E$ is isometrically order isomorphic to a subspace of $C(\Omega)$, for some compact Hausdorff space $\Omega$.
\end{theorem K}

This representation is given by a concrete construction, and Kadison \cite{Kadison-representation} showed that his construction contains various known representation theorems as special cases, including results for real Banach spaces, real Banach algebras, and Archimedean Riesz spaces with an order unit.
The continuing relevance of Kadison's representation theorem can be seen, for instance, in the theory of operator systems (\cite{Paulsen-Tomforde}, \cite{Paulsen-Todorov-Tomforde}).

A more general result is Schaefer's representation theorem, which gives necessary and sufficient conditions for an ordered locally convex space $E$ to be topologically order isomorphic to a subspace of $\Ck(\Omega)$ for some locally compact Hausdorff space $\Omega$, where $\Ck(\Omega)$ denotes the space $C(\Omega)$ equipped with the topology of \emph{compact convergence} (i.e.~uniform convergence on compact subsets).

\begin{theorem S}[{Schaefer \cite[Theorem 5.1]{Schaefer-I}; see also \cite[Theorem V.4.4]{Schaefer}}]
	\label{thm:Schaefer-representation}
	Let $E$ be a \textup(real\textup) ordered locally convex space with positive cone $E_+$. Then there exists a locally compact Hausdorff space $\Omega$ such that $E$ is topologically order isomorphic to a subspace of $\Ck(\Omega)$ if and only if the positive cone $E_+$ is closed and normal.
\end{theorem S}

Kadison's and Schaefer's representation theorems both result in a topological order isomorphism of $E$ with a subspace of a space of continuous functions.
These results are very strong, but at the same time rather restrictive.
By comparison, the Gelfand representation of a commutative Banach algebra is not always a topological embedding (or even injective), so there appears to be room for a more general representation theory of ordered topological vector spaces.

In this paper, we will study more general representations of ordered topological vector spaces by relaxing the requirements on the representing operator $E \to \Ck(\Omega)$; instead of requiring a topological order embedding, we look for representations that are continuous and positive.

Our main result, \autoref{thm:intro:pos:semisimple} below, is a characterization of the class of semisimple ordered topological vector spaces, consisting of those spaces $E$ that admit an injective representation $E \to \Ck(\Omega)$ of the aforementioned kind (continuous and positive).
We will show that semisimplicity is characterized by many equivalent geometric, algebraic, and topological conditions.
Furthermore, we will give similar equivalent criteria for the existence of a continuous and \emph{bipositive} representation $E \to \Ck(\Omega)$, so that the positive cone of $E$ coincides with the \emph{pullback cone} inherited from $\Ck(\Omega)$.

Although Kadison's and Schaefer's representation theorems give the impression that order units and/or normality are somehow essential to the process, our results show that these requirements are only needed in order to obtain a topological embedding.
To obtain an injective positive (resp.~bipositive) representation $E \to \Ck(\Omega)$, all that is needed is that $E_+$ is contained in (resp.~is itself) a weakly closed proper cone.

Semisimple ordered topological vector spaces (as defined in this paper) are closely related to Schaefer's \emph{regularly ordered spaces}, which are given by the following equivalent conditions.

\begin{theorem S}[{Schaefer \cite[Proposition 1.7]{Schaefer-I}; see also \cite[Proposition II.1.29]{Peressini}}]
	\label{thm:Schaefer-regular}
	For a convex cone $E_+$ in a real vector space $E$, the following are equivalent:
	\begin{enumerate}[label=(\roman*)]
		\item The algebraic dual cone $E_+\algdual$ separates points on $E$;
		\item The closure of $E_+$ in the finest locally convex topology on $E$ is a proper cone;
		\item There is a locally convex topology on $E$ for which $\overline{E_+}$ is a proper cone;
		\item There is a locally convex topology on $E$ for which $E_+$ is normal.
	\end{enumerate}
\end{theorem S}

(For Riesz spaces, see also \cite[Theorem 23.15]{Kelley-Namioka}.)
If $E_+$ satisfies any one \textup(and therefore all\textup) of the properties from \autoref{thm:Schaefer-regular}, then $E_+$ is called \emph{regular}.
Semisimplicity is the topological analogue of regularity, and the equivalent characterizations of semisimple cones lead to new equivalent conditions for regularity, in addition to the criteria from \autoref{thm:Schaefer-regular}.

An application of semisimplicity will be given in the follow-up paper \cite{order-unitizations}, where we will show that a \emph{normed} ordered vector space $E$ admits an \emph{Archimedean order unitization} if and only if $E_+$ is semisimple.

\subsection{Main results}
Throughout this paper, all topological vector spaces will be over the real numbers.
The topological dual of a topological vector space $E$ is denoted $E\topdual$, and the weak\nobreakdash-$*$ topology on $E\topdual$ is denoted $\weakstar$.
For additional notation, see \mysecref{sec:notation}.

\subsubsection{A general theory of representations}
Let $\Omega$ be a topological space, and let $C(\Omega)$ denote the space of all continuous functions $\Omega \to \R$. Furthermore, let $\Cp(\Omega)$ and $\Ck(\Omega)$ denote the space $C(\Omega)$ equipped with the locally convex topologies of pointwise and compact convergence, respectively. If $\Omega$ is compact, then $\Ck(\Omega)$ carries the usual norm topology.

A \emph{representation} of the topological vector space $E$ is a continuous linear map $E \to F$, where $F$ is a locally convex space of functions. In this paper, $F$ will always be $\Cp(\Omega)$ or $\Ck(\Omega)$ for some topological space $\Omega$. When studying positive representations, we will understand these function spaces to be equipped with the cone of everywhere non-negative functions.

In \mysecref{sec:representation-lemma}, we set up a general theory of representations $E \to \Cp(\Omega)$. The following lemma, which we believe to be of independent interest, sums up the main results of \mysecref{sec:representation-lemma}.

\begin{AlphLemma}
	\label{lem:intro:representation-lemma}
	Let $E$ be a real topological vector space, and let $\Omega$ be a topological space.
	\begin{enumerate}[label=(\alph*),series=intro-representations]
		\item\label{itm:intro:repr-correspondence} There is a bijective correspondence $T \mapsto T^\ECt$ between representations $T : E \to \Cp(\Omega)$ and continuous functions $T^\ECt : \Omega \to E_\weakstar\topdual$, such that $T(x)(\omega) = T^\ECt(\omega)(x)$.
	\end{enumerate}
	Under the correspondence from \ref{itm:intro:repr-correspondence}, the following properties of a representation $T : E \to \Cp(\Omega)$ depend only on the image of $T^\ECt$ in $E_\weakstar\topdual$:
	\begin{enumerate}[resume*=intro-representations]
		\item A representation $T : E \to \Cp(\Omega)$ is injective if and only if $T^\ECt[\Omega]$ separates points on $E$;
		
		\item A representation $T : E \to \Cp(\Omega)$ is continuous as a map $E \to \Ck(\Omega)$ if and only if $T^\ECt[K]$ is equicontinuous for every compact subset $K \subseteq \Omega$. This is automatically the case if $E$ is a barrelled locally convex space \textup(in particular, if $E$ is a Fr\'echet space\textup).
	\end{enumerate}
	Furthermore, if $E$ is a preordered topological vector space with positive cone $E_+$, then:
	\begin{enumerate}[resume*=intro-representations]
		\item A representation $T : E \to \Cp(\Omega)$ is positive if and only if $T^\ECt[\Omega] \subseteq E_+\topdual$;
		
		\item A representation $T : E \to \Cp(\Omega)$ is bipositive if and only if $E_+$ is weakly closed and $E_+\topdual$ is the weak\nobreakdash-$*$ closed convex cone generated by $T^\ECt[\Omega]$.
	\end{enumerate}
\end{AlphLemma}

\smallskip
Proofs of the statements in \autoref{lem:intro:representation-lemma} will be given in \mysecref{subsec:transpose}.

The correspondence between representations $E \to \Cp(\Omega)$ and continuous functions $\Omega \to E_\weakstar\topdual$ greatly simplifies the task of choosing an appropriate representation of $E$.
Common choices of representations $E \to \Cp(\Omega)$ or $E \to \Ck(\Omega)$ will be discussed in \mysecref{subsec:explicit-constructions}.

\subsubsection{Equivalent definitions of semisimplicity}
The main goal of this paper is to classify all (pre)ordered topological vector spaces that admit an injective positive representation to a space of functions.
Using \autoref{lem:intro:representation-lemma} and a few standard results on ordered topological vector spaces, we obtain a wealth of equivalent geometric, algebraic, and topological characterizations.

\begin{AlphTheorem}[Criteria for semisimplicity]
	\label{thm:intro:pos:semisimple}
	For a convex cone $E_+ $ in a real topological vector space $E$, the following are equivalent:
	\begin{enumerate}[label=(\roman*)]
		\item\label{itm:intro:pos:dual-cone-separates-points} The topological dual cone $E_+\topdual$ separates points on $E$;
		\item\label{itm:intro:pos:weak-closure-is-proper} The topological dual space $E\topdual$ separates points and the weak closure $\overline{E_+}^{\,\weak}$ is a proper cone;
		\item\label{itm:intro:pos:intersection-closed-hyperplanes} The intersection of all closed supporting hyperplanes of $E_+$ is $\{0\}$;
		\item\label{itm:intro:pos:coarser-normality} There is a weaker \textup(i.e.~coarser\textup) locally convex topology on $E$ for which $E_+$ is normal;
		\item\label{itm:intro:pos:Cp-representation} There is a topological space $\Omega$ and an injective positive representation $E \to \Cp(\Omega)$;
		\item\label{itm:intro:pos:Ck-representation} There is a locally compact Hausdorff space $\Omega$ and an injective positive representation $E \to \Ck(\Omega)$.
	\end{enumerate}
	Furthermore, if $E$ is a normed space, then the topology in \myref{itm:intro:pos:coarser-normality} can be taken normable, and the space $\Omega$ in \myref{itm:intro:pos:Ck-representation} can be taken compact.
\end{AlphTheorem}

A cone satisfying any one (and therefore all) of the properties of \autoref{thm:intro:pos:semisimple} will be called (\emph{topologically}) \emph{semisimple}, in analogy with semisimple (commutative) Banach algebras.
Note that the semisimplicity of $E_+$ depends only on $E_+$ and on the dual pair $\langle E,E\topdual\rangle$, not on the topology of $E$ (use property \myref{itm:intro:pos:dual-cone-separates-points} or \myref{itm:intro:pos:weak-closure-is-proper}).

\autoref{thm:intro:pos:semisimple} can also be applied in a purely algebraic setting (i.e.~if $E$ is a vector space with no topology). Clearly $E_+$ is regular if and only if $E_+$ is semisimple for the $\sigma(E,E\algdual)$-topology, or any other topology compatible with the dual pair $\langle E,E\algdual\rangle$. This shows that semisimplicity is the topological analogue of regularity.\hair%
\footnote{The use of these terms is somewhat arbitrary. What is traditionally called \emph{regular} could just as well be called (\emph{algebraically}) \emph{semisimple}. We justify our terminology by drawing parallels with semisimple commutative Banach algebras, but there are similar parallels between regularly ordered spaces and Jacobson semisimple rings. (See \autoref{rmk:radicals} below.)}
Furthermore, \autoref{thm:intro:pos:semisimple} provides the following additional characterizations of regularity.

\begin{AlphCorollary}
	\label{cor:intro:pos:regular}
	For a convex cone $E_+$ in a real vector space $E$, the following are equivalent:
	\begin{enumerate}[label=(\roman*)]
		\item $E_+$ is regular \textup(in the sense of \autoref{thm:Schaefer-regular}\textup);
		\item The intersection of all supporting hyperplanes of $E_+$ is $\{0\}$;
		\item There is a topological space $\Omega$ and an injective positive linear map $E \to C(\Omega)$;
		\item There is a locally compact Hausdorff space $\Omega$ and an injective positive linear map $E \to C(\Omega)$.
	\end{enumerate}
\end{AlphCorollary}

\smallskip
\autoref{thm:intro:pos:semisimple} and \autoref{cor:intro:pos:regular} will be proved in \mysecref{sec:positive}.

\subsubsection{Bipositive representations}
The preceding results show exactly which spaces admit an injective positive representation $E \to \Cp(\Omega)$ or $E \to \Ck(\Omega)$. However, sometimes it is desirable to have something stronger, namely a bipositive representation, so that the positive cone of $E$ coincides with the cone inherited from the function space.

In \mysecref{sec:bipositive}, we will show that bipositive versions of \autoref{thm:intro:pos:semisimple} and \autoref{cor:intro:pos:regular} can be obtained by adding the requirement that $E_+$ is weakly closed.
These results can be seen as a relaxation of Schaefer's representation theorem (\autoref{thm:Schaefer-representation} above): we no longer require that $E_+$ is normal, but we only get a representation that is continuous (not a topological embedding).

\subsubsection{Hereditary properties and examples of semisimplicity}
In \mysecref{sec:stability}, we will show that semisimplicity is preserved by subspaces, products, and direct sums, but not by (proper) quotients or completions.
For semisimplicity in tensor products, see \cite{ordered-tensor-products-i}.

In \mysecref{sec:examples}, we give various examples and counterexamples related to semisimplicity and regularity.

\section{Notation}
\label{sec:notation}
Throughout this article, all vector spaces are over the real numbers, and all topological vector spaces are Hausdorff. Algebraic duals will be denoted as $E\algdual$ and topological duals as $E\topdual$. The weak\nobreakdash-$*$ topology on $E\topdual$ will be denoted as $E_\weakstar\topdual$. Similarly, if $E\topdual$ separates points on $E$, then the weak topology on $E$ will be denoted as $E_\weak$.

\subsubsection{Convex cones and preordered vector spaces}
Let $E$ be a real vector space. A (\emph{convex}) \emph{cone}\hair\footnote{Some authors call this a \emph{wedge}, and reserve the term \emph{cone} for what we call a \emph{proper cone}.} is a non-empty subset $\mywedge \subseteq E$ satisfying $\mywedge + \mywedge \subseteq \mywedge$ and $\lambda \mywedge \subseteq \mywedge$ for all $\lambda \in \R_{\geq 0}$. If $\mywedge$ is a convex cone, then $\lineal(\mywedge) := \mywedge \cap -\mywedge$ is a linear subspace of $E$, called the \emph{lineality space} of $\mywedge$. We say that $\mywedge$ is \emph{proper} if $\lineal(\mywedge) = \{0\}$.

If $M \subseteq E$ is a subset, then we denote by $\cone(M)$ the \emph{convex cone generated by $M$}; that is:
\[ \cone(M) := \left\{\left.\sum_{i=1}^k \lambda_i x_i \ \right\rvert \ k \in \N_0,\ \lambda_1,\ldots,\lambda_k \in \R_{\geq 0},\ x_1,\ldots,x_k \in M\right\}. \]

Convex cones in $E$ are in bijective correspondence with linear preorders on $E$. Under this correspondence, a convex cone is proper if and only if the associated preorder is a partial order.

In what follows, whenever we study a vector space $E$ and a convex cone $E_+ \subseteq E$, we will understand $E$ to be the preordered vector space with positive cone $E_+$.

\subsubsection{Positive and bipositive linear maps}
If $E$, $F$ are vector spaces and $E_+ \subseteq E$, $F_+ \subseteq F$ are convex cones, then a linear map $T : E \to F$ is called \emph{positive} if $T[E_+] \subseteq F_+$, and \emph{bipositive} if $E_+ = T^{-1}[F_+]$.

If $E_+$ is a proper cone, then a bipositive map $T : E \to F$ is automatically injective, because $\ker(T) = T^{-1}[\{0\}] \subseteq T^{-1}[F_+] = E_+$ is a subspace contained in $E_+$, which must therefore be $\{0\}$. Note however that this is not true in general: if $E/\lineal(E_+)$ is equipped with the quotient cone, then the canonical map $E \to E/\lineal(E_+)$ is bipositive, but not injective unless $E_+$ is a proper cone.

\subsubsection{Dual cones}
If $E_+ \subseteq E$ is a convex cone, then its \emph{algebraic dual cone} $E_+\algdual \subseteq E\algdual$ is the set of all positive linear functionals, and its \emph{topological dual cone} $E_+\topdual \subseteq E\topdual$ is the set of all \emph{continuous} positive linear functionals. If we equip $E\topdual$ with the weak\nobreakdash-$*$ topology, then the dual cone of the dual cone $E_+\topdual \subseteq E_\weakstar\topdual$ is the \emph{bipolar cone}
\[ E_+\topdualdual := \big\{x \in E \, : \, \langle x,\varphi\rangle \geq 0 \ \text{for all $\varphi \in E_+\topdual$}\big\}. \]
Using the (one-sided) bipolar theorem, one easily shows that $E_+\topdualdual = \overline{E_+}^{\,\weak}$. 
Furthermore, if $E$ is locally convex, then the weak closure and original closure of a convex set coincide, so we have
\[ E_+\topdualdual = \overline{E_+}^{\,\weak} = \overline{E_+} \qquad \text{($E$ locally convex, $E_+ \subseteq E$ a convex cone)}. \]

\section{Representations of topological vector spaces}
\label{sec:representation-lemma}
Before proving our main theorem, we study the general structure of representations.
Recall that we defined a \emph{representation} to be any continuous linear operator $E \to \Cp(\Omega)$ or $E \to \Ck(\Omega)$.
In this section, we will establish a bijective correspondence between representations $E \to \Cp(\Omega)$ and continuous maps $\Omega \to E_\weakstar\topdual$. Under this correspondence, the important properties of the representation (e.g.~injectivity, positivity) only depend on the image of $\Omega$ in $E\topdual$ (\autoref{lem:representation}).

Throughout this section, we assume that $E$ is a real\hair%
\footnote{The results from this section are also true over the complex field, but we have no use for this. For notational simplicity, we will stick to the assumption (made throughout this article) that the base field is $\R$.}
topological vector space. We do not assume that the topological dual $E\topdual$ separates points on $E$. Nevertheless, we denote the natural bilinear map $E \times E\topdual \to \R$ by $\langle \:\cdot\: , \:\cdot\: \rangle$, and we denote the weak\nobreakdash-$*$ topology on $E\topdual$ by $E_\weakstar\topdual$. (Note that $E$ automatically separates points on $E\topdual$, so $E_\weakstar\topdual$ is Hausdorff.)

Let $\Omega$ be a topological space.
For $\omega \in \Omega$, let $\pi_\omega : \Cp(\Omega) \to \R$ denote the point evaluation $f \mapsto f(\omega)$. Note that we have $\Cp(\Omega) \subseteq \R^\Omega$, and that the topology of $\Cp(\Omega)$ coincides with the subspace topology inherited from the product topology of $\R^\Omega$. In particular, a function $\phi : Y \to \Cp(\Omega)$ is continuous if and only if $\pi_\omega \circ \phi$ is continuous for all $\omega \in \Omega$.

\subsection{The transpose of a representation}
\label{subsec:transpose}

Given sets $A$ and $B$ and a function $f : A \times B \to \R$, we denote by $f^\ABt : B \times A \to \R$ the \emph{transpose} $f^\ABt(b)(a) = f(a)(b)$.
As a particular case, if $T : E \to \Cp(\Omega)$ is a representation, then $T$ defines a function $E \times \Omega \to \R$, which has a transpose $T^\ECt : \Omega \times E \to \R$.
Similarly, if $f : \Omega \to E_\weakstar\topdual$ is a continuous function, then $f$ defines a function $\Omega \times E \to \R$, which has a transpose $f^\OEt : E \times \Omega \to \R$.
Our goal is to show that transposition on $E \times \Omega$ restricts to a bijection between the representations $E \to \Cp(\Omega)$ and the continuous functions $\Omega \to E_\weakstar\topdual$.

If $T : E \to \Cp(\Omega)$ is a representation, then we interpret $T^\ECt$ as the map $\Omega \to E\topdual$, $\omega \mapsto \pi_\omega \circ T$, which is easily seen to be well-defined ($\pi_\omega \circ T$ is a continuous linear functional on $E$).
Similarly, if $f : \Omega \to E_\weakstar\topdual$ is a continuous function, then we interpret $f^\OEt$ as the map $E \to \Cp(\Omega)$, $x \mapsto (\omega \mapsto \langle x, f(\omega) \rangle)$. To see that $f^\OEt$ is well-defined, note that every $x \in E$ defines a weak\nobreakdash-$*$ continuous linear functional $\hat x := \langle x , \:\cdot\: \rangle$ on $E\topdual$, so $f^\OEt(x) = \hat x \circ f$ is a continuous map $\Omega \to \R$.

\begin{proposition}
	Let $E$ be a real topological vector space and let $\Omega$ be a topological space.
	\begin{enumerate}[label=(\alph*)]
		\item If $T : E \to \Cp(\Omega)$ is a representation, then $T^\ECt : \Omega \to E_\weakstar\topdual$ is continuous, and $(T^\ECt)^\OEt = T$.
		
		\item If $f : \Omega \to E_\weakstar\topdual$ is continuous, then $f^\OEt : E \to \Cp(\Omega)$ is a representation, and $(f^\OEt)^\ECt = f$.
	\end{enumerate}
\end{proposition}
\begin{proof}\ \par
	\begin{enumerate}[label=(\alph*)]
		\item Since $E_\weakstar\topdual$ carries the initial topology for the family $\{\hat x \, : \, x \in E\}$, in order to prove that $T^\ECt : \Omega \to E_\weakstar\topdual$ is continuous, it suffices to show that $\hat x \circ T^\ECt$ is continuous for every $x \in E$.
		But $\hat x \circ T^\ECt$ is given by $\omega \mapsto \langle x , \pi_\omega \circ T\rangle = (\pi_\omega \circ T)(x) = T(x)(\omega)$, which is continuous since $T(x) \in \Cp(\Omega)$. This also shows that $(T^\ECt)^\OEt(x) = \hat x \circ T^\ECt = T(x)$, hence $(T^\ECt)^\OEt = T$.
		
		\item It is easy to see that $f^\OEt$ is linear. To prove continuity, it suffices to show that $\pi_\omega \circ f^\OEt$ is continuous for every $\omega \in \Omega$, since $\Cp(\Omega)$ carries the initial topology for the family $\{\pi_\omega \, : \, \omega \in \Omega\}$. But $\pi_\omega \circ f^\OEt$ is given by $x \mapsto \langle x , f(\omega)\rangle$, so that $\pi_\omega \circ f^\OEt = f(\omega)$ is a continuous linear functional on $E$. This also shows that $(f^\OEt)^\ECt(\omega) = \pi_\omega \circ f^\OEt = f(\omega)$, hence $(f^\OEt)^\ECt = f$.
		\qedhere
	\end{enumerate}
\end{proof}
\begin{corollary}
	Let $E$ be a real topological vector space, and let $\Omega$ be a topological space. Then the operations $T \mapsto T^\ECt$ and $f \mapsto f^\OEt$ define a bijective correspondence between the representations $E \to \Cp(\Omega)$ and the continuous functions $\Omega \to E_\weakstar\topdual$.
\end{corollary}
\begin{remark}
	Going back to our initial interpretation of the transpose, one easily shows that a function $f : E \times \Omega \to \R$ is given by a representation $E \to \Cp(\Omega)$ (or by a continuous function $\Omega \to E_\weakstar\topdual$) if and only if $f$ is separately continuous (in both coordinates) and linear in $E$. This is an asymmetric analogue of the fact that the separately continuous bilinear forms $E \times F \to \R$ correspond to the continuous linear operators $E \to F_\weakstar\topdual$ or $F \to E_\weakstar\topdual$ (e.g.~\cite[\S 40.1.(2')]{Kothe-II}).
\end{remark}

The following lemma shows that the most important qualitative properties of a representation $T : E \to \Cp(\Omega)$ depend only on the image of $T^\ECt$ in $E\topdual$.

\begin{lemma}
	\label{lem:representation}
	Let $E$ be a real topological vector space, $\Omega$ a topological space, and $T : E \to \Cp(\Omega)$ a representation. Then:
	\begin{enumerate}[label=(\alph*),series=representations]
		\item\label{itm:repr:injective} $T$ is injective if and only if $T^\ECt[\Omega]$ separates points on $E$;
		
		\item\label{itm:repr:Ck-continuous} $T$ is continuous as a map $E \to \Ck(\Omega)$ if and only if $T^\ECt[K]$ is equicontinuous for every compact subset $K \subseteq \Omega$.
	\end{enumerate}
	If, additionally, $E$ is preordered with positive cone $E_+$, then:
	\begin{enumerate}[label=(\alph*),resume=representations]
		\item\label{itm:repr:positive} $T$ is positive if and only if $T^\ECt[\Omega] \subseteq E_+\topdual$;
		
		\item\label{itm:repr:bipositive} $T$ is bipositive if and only if $E_+$ is weakly closed and $E_+\topdual$ is the weak\nobreakdash-$*$ closed convex cone generated by $T^\ECt[\Omega]$.
	\end{enumerate}
\end{lemma}
\begin{proof}\ \par
	\begin{enumerate}[label=(\alph*)]
		\item This follows from the identity
		\[ \ker(T) = \big\{x \in E \, : \, T(x)(\omega) = 0\ \text{for all $\omega \in \Omega$}\big\} = \bigcap_{\omega \in \Omega} \ker(\pi_\omega \circ T) = \bigcap_{\omega \in \Omega} \ker(T^\ECt(\omega)). \]
		
		\item The topology of $\Ck(\Omega)$ is given by the family of seminorms $\{\lVert \:\cdot\: \rVert_K \, : \, K \subseteq \Omega\ \text{compact}\}$, where $\lVert g\rVert_K := \max_{k \in K} |g(k)|$. Thus, $T$ is continuous as a map $E \to \Ck(\Omega)$ if and only if for every compact subset $K \subseteq \Omega$ and every real number $\varepsilon > 0$ there is a $0$-neighbourhood $U_{K,\varepsilon} \subseteq E$ such that
		\begin{align}
			\lVert T(u)\rVert_K < \varepsilon,\qquad &\text{for all $u\in U_{K,\varepsilon}$}. \label{eqn:repr:pf:Ck-continuous}
			\intertext{Since $K$ is compact, the maximum $\max_{k\in K} |T(u)(k)|$ is attained, so \eqref{eqn:repr:pf:Ck-continuous} is equivalent to}
			\langle u , T^\ECt(k) \rangle = T(u)(k) \in (-\varepsilon,\varepsilon),\qquad &\text{for all $u\in U_{K,\varepsilon}$ and all $k \in K$}. \tag{\ref*{eqn:repr:pf:Ck-continuous}'} \label{eqn:repr:pf:equicontinuous}
		\end{align}
		The result follows, since $T^\ECt[K]$ is equicontinuous if and only if for every real number $\varepsilon > 0$ there is a $0$-neighbourhood $U_{K,\varepsilon} \subseteq E$ such that \eqref{eqn:repr:pf:equicontinuous} is met.

		\item Since the ordering of $\Cp(\Omega)$ is pointwise, $T$ is positive if and only if $\pi_\omega \circ T$ is positive for every $\omega \in \Omega$. Since $T^\ECt(\omega) = \pi_\omega \circ T$ (by definition), the claim follows.
		
		\item We have
		\begin{align*}
			T^{-1}[\Cp(\Omega)_+] &= \big\{x \in E \, : \, T(x)(\omega) \geq 0 \ \text{for all $\omega \in \Omega$}\big\}\\\noalign{\smallskip}
			&= \big\{x \in E \, : \, \langle x , T^\ECt(\omega) \rangle \geq 0\ \text{for all $\omega \in \Omega$}\big\}\\\noalign{\smallskip}
			&= \big\{x \in E \, : \, \langle x , \varphi \rangle \geq 0 \ \text{for all $\varphi \in \cone(T^\ECt[\Omega])$}\big\},
		\end{align*}
		for if $x \in E$ is non-negative on $T^\ECt[\Omega]$ then it is automatically non-negative on $\cone(T^\ECt[\Omega])$.
		It follows that $T^{-1}[\Cp(\Omega)_+] \subseteq E$ is the dual cone of $\cone(T^\ECt[\Omega]) \subseteq E_\weakstar\topdual$. Consequently, by the bipolar theorem, we have $E_+ = T^{-1}[\Cp(\Omega)_+]$ if and only if $E_+$ is weakly closed and
		\[ E_+\topdual = \overline{\cone(T^\ECt[\Omega])}^{\,\weakstar}. \qedhere \]
	\end{enumerate}
\end{proof}

\begin{remark}
	In \myautoref{lem:representation}{itm:repr:Ck-continuous}, note that $T^\ECt[\Omega]$ is automatically weak\nobreakdash-$*$ compact if $K \subseteq \Omega$ is compact. Equicontinuity is only slightly stronger than weak\nobreakdash-$*$ compactness. In fact, for every subset of $E\topdual$ one has
	\[ \text{equicontinuous} \implies \text{weak-$*$ relatively compact} \implies \text{weak-$*$ bounded}. \]
	The first implication follows from the Alaoglu--Bourbaki theorem (see \cite[Corollary III.4.3]{Schaefer}), and the second follows since every (subset of a) compact set in a topological vector space is bounded (cf.~\cite[Theorem 1.15(b)]{Rudin}). 
	
	A locally convex space $E$ is called \emph{barrelled} if every weak\nobreakdash-$*$ bounded subset of $E\topdual$ is equicontinuous, so that all three of the aforementioned classes coincide. The class of barrelled spaces includes all Fr\'echet spaces (in particular, all Banach spaces), but not all normed spaces.
	
	If $E$ is barrelled, then every $T^\ECt[K]$ ($K \subseteq \Omega$ compact) is weak\nobreakdash-$*$ compact and therefore equicontinuous, so here \myautoref{lem:representation}{itm:repr:Ck-continuous} has the following immediate consequence:
\end{remark}

\begin{corollary}
	If $E$ is \textup(locally convex and\textup) barrelled, then every representation $T : E \to \Cp(\Omega)$ is automatically continuous $E \to \Ck(\Omega)$.
\end{corollary}

This is essentially a special case of the generalized Banach--Steinhaus theorem for barrelled spaces (e.g.~\cite[Theorem III.4.2]{Schaefer} or \cite[Theorem 4.16]{Osborne}).

\subsection{Explicit constructions of representations}
\label{subsec:explicit-constructions}
Assume now that only a topological vector space $E$ is given. We discuss some of the common choices of a space $\Omega$ and a representation $E \to \Cp(\Omega)$ or $E \to \Ck(\Omega)$.

\subsubsection{Pointwise convergence on a subset $\Omega \subseteq E\topdual$ with the relative weak\protect\nobreakdash-$*$ topology}
\autoref{lem:representation} shows that the properties of the representation only depend on the image of $\Omega$ in $E\topdual$. So if only a representation $E \to \Cp(\Omega)$ is desired, then one might simply take $\Omega$ to be a subset of $E\topdual$ with the relative weak\nobreakdash-$*$ topology.

\subsubsection{Compact convergence on a subset $\Omega_\discr \subseteq E\topdual$ with the discrete topology}
A simple way to obtain representations $E \to \Ck(\Omega)$ with $\Omega$ locally compact is the following: take some subset $\Omega \subseteq E\topdual$, and let $\Omega_\discr$ denote $\Omega$ equipped with the discrete topology. Then $\Omega_\discr$ is locally compact, and the compact subsets of $\Omega_\discr$ are precisely the finite sets, so $E \to \Ck(\Omega_\discr)$ is continuous by \myautoref{lem:representation}{itm:repr:Ck-continuous}. However, $\Ck(\Omega_\discr)$ is simply $\R^\Omega$ with the product topology, so this representation is not particularly useful. In fact, it factors as $E \to \Cp(\Omega_\weakstar) \to \Ck(\Omega_\discr) \cong \R^\Omega$.

\subsubsection{Compact convergence on a disjoint union of weak\protect\nobreakdash-$*$ closed equicontinuous sets}
To obtain representations $E \to \Ck(\Omega)$ with a coarser (non-discrete) topology on $\Omega$ and a finer (non-pointwise) topology on $\Ck(\Omega)$, it is more common to use the following method.

\begin{lemma}
	\label{lem:Ck-representation}
	Let $\{N_\alpha\}_{\alpha \in A}$ be a collection of weak\nobreakdash-$*$ closed equicontinuous subsets of $E\topdual$. Then the disjoint union $\Omega := \bigsqcup_{\alpha \in A} N_\alpha$ is locally compact, every compact subset $K \subseteq \Omega$ is contained in the disjoint union of finitely many $N_\alpha$, and the natural map $E \to \Ck(\Omega)$ is continuous.
\end{lemma}
\begin{proof}
	It follows from the Alaoglu--Bourbaki theorem that equicontinuous subsets in $E\topdual$ are relatively weak\nobreakdash-$*$ compact (e.g.~\cite[Corollary III.4.3]{Schaefer}), so every $N_\alpha$ is weak\nobreakdash-$*$ compact.
	
	For every $\alpha \in A$, let $\Omega_\alpha \subseteq \Omega$ denote the image of the canonical embedding $N_\alpha \to \Omega$. Then $\Omega_\alpha$ is a compact neighbourhood of every point $\omega \in \Omega_\alpha$, so $\Omega$ is locally compact. Furthermore, $\Omega = \bigcup_{\alpha \in A} \Omega_\alpha$ is an open cover of $\Omega$, so every compact subset $K \subseteq \Omega$ is contained in the union of finitely many $\Omega_\alpha$.
	
	The corresponding representation is $f^\OEt$, where $f : \Omega \to E_\weakstar\topdual$ is the natural map identifying every $\Omega_\alpha \subseteq \Omega$ with $N_\alpha \subseteq E\topdual$. If $K \subseteq \Omega$ is compact, then $f[K]$ is contained in the union of finitely many $N_\alpha$, hence equicontinuous. It follows from \myautoref{lem:representation}{itm:repr:Ck-continuous} that $E \to \Ck(\Omega)$ is continuous.
\end{proof}

If the only objective is to obtain a locally compact space $\Omega$ and a representation $E \to \Ck(\Omega)$, then the simpler construction $E \to \Ck(\Omega_\discr) \cong \R^\Omega$ outlined above can be used.
The real benefit of the more advanced construction from \autoref{lem:Ck-representation} is this:
\begin{proposition}
	\label{prop:top-emb}
	Let $\{N_\alpha\}_{\alpha \in A}$ be as in \autoref{lem:Ck-representation}. If the topology of $E$ coincides with the \textup(locally convex\textup) topology of uniform convergence on the sets $\{N_\alpha\}_{\alpha \in A}$, then the corresponding representation $E \to \Ck(\Omega)$ is a topological embedding.
\end{proposition}

In other words, in this case $E$ can be thought of as a subspace of $\Ck(\Omega)$ equipped with the subspace topology. Various classical representation theorems start by choosing a collection $\{N_\alpha\}_{\alpha \in A}$ with this property (see e.g.~\cite[Theorem V.4.4]{Schaefer} and \cite[\S 20.10.(3)]{Kothe-I}).
We will not use this, so the proof of \autoref{prop:top-emb} is omitted.

Classical representations of normed spaces (including Kadison's representation and the Gelfand representation) take $\{N_\alpha\}_{\alpha \in A}$ to be a one-element collection $\{N\}$, where $N \subseteq \ball{E\topdual}$ is a weak\nobreakdash-$*$ closed subset of the closed unit ball $\ball{E\topdual} \subseteq E\topdual$. Exactly which subset $N \subseteq \ball{E\topdual}$ is chosen depends on the context.

\section{Semisimple cones and positive representations}
\label{sec:positive}
In this section, we prove the main theorem. For this we use the following proposition.

\begin{proposition}
	\label{prop:supporting-hyperplanes}
	In a preordered vector space $E$, the supporting hyperplanes of $E_+$ are precisely the kernels of the non-zero positive linear functionals $E \to \R$.
\end{proposition}
\begin{proof}
	If $\varphi$ is a non-zero positive linear functional, then $E_+$ is contained in $\{x \in E \, : \, \varphi(x) \geq 0\}$, and one has $0 \in \ker(\varphi) \cap E_+$, so $\ker(\varphi)$ is a supporting hyperplane of $E_+$.
	
	Conversely, let $H \subseteq E$ be an (affine) supporting hyperplane of $E_+$. Choose $\varphi \in E\algdual \setminus \{0\}$ and $\alpha \in \R$ such that $H = \{x \in E \, : \, \varphi(x) = \alpha\}$ and $E_+ \subseteq \{x \in E \, : \, \varphi(x) \geq \alpha\}$, and choose $x_0 \in E_+ \cap H$. For all $n \in \N_0$ we have $nx_0 \in E_+$, and therefore $n\alpha = n\varphi(x_0) = \varphi(nx_0) \geq \alpha$. Plugging in $n = 0$ yields $\alpha \leq 0$, while plugging in $n = 2$ yields $\alpha \geq 0$, so we must have $\alpha = 0$. It follows that $H = \ker(\varphi)$ and that $\varphi$ is a non-zero positive linear functional.
\end{proof}

\begin{corollary}
	\label{cor:closed-supporting-hyperplanes}
	In a preordered topological vector space $E$, the closed supporting hyperplanes of $E_+$ are precisely the kernels of the non-zero, positive, continuous linear functionals $E \to \R$.
\end{corollary}
\begin{proof}
	This follows from \autoref{prop:supporting-hyperplanes} and the fact that a non-zero linear functional $E \to \R$ is continuous if and only if $\ker(E)$ is closed (see \cite[Theorem 1.18]{Rudin}).
\end{proof}

We now have all the ingredients for the classification of semisimple ordered topological vector spaces.

\begin{theorem}[Criteria for semisimplicity]
	\label{thm:pos:semisimple}
	For a convex cone $E_+ $ in a real topological vector space $E$, the following are equivalent:
	\begin{enumerate}[label=(\roman*)]
		\item\label{itm:pos:dual-cone-separates-points} The topological dual cone $E_+\topdual$ separates points on $E$;
		\item\label{itm:pos:weak-closure-is-proper} The topological dual space $E\topdual$ separates points and the weak closure $\overline{E_+}^{\,\weak}$ is a proper cone;
		\item\label{itm:pos:intersection-closed-hyperplanes} The intersection of all closed supporting hyperplanes of $E_+$ is $\{0\}$;
		\item\label{itm:pos:coarser-normality} There is a weaker \textup(i.e.~coarser\textup) locally convex topology on $E$ for which $E_+$ is normal;
		\item\label{itm:pos:Cp-representation} There is a topological space $\Omega$ and an injective positive representation $E \to \Cp(\Omega)$;
		\item\label{itm:pos:Ck-representation} There is a locally compact Hausdorff space $\Omega$ and an injective positive representation $E \to \Ck(\Omega)$.
	\end{enumerate}
	Furthermore, if $E$ is a normed space, then the topology in \myref{itm:pos:coarser-normality} can be taken normable, and the space $\Omega$ in \myref{itm:pos:Ck-representation} can be taken compact.
\end{theorem}
\begin{proof}
	$\myref{itm:pos:dual-cone-separates-points} \Longleftrightarrow \myref{itm:pos:intersection-closed-hyperplanes}$. By \autoref{cor:closed-supporting-hyperplanes}, we have
	\[ \bigcap_{\substack{H\subseteq E\ \text{closed}\\\text{supporting hyplerplane}}} \hspace*{-7mm} H \hspace{7mm} = \ \bigcap_{\varphi \in E_+\topdual \setminus \{0\}} \ker(\varphi) \ =\  \bigcap_{\varphi \in E_+\topdual} \ker(\varphi), \]
	so the intersection of all closed supporting hyperplanes of $E_+$ is equal to $\{0\}$ if and only if $E_+\topdual$ separates points on $E$.
	
	$\myref{itm:pos:dual-cone-separates-points} \Longrightarrow \myref{itm:pos:Ck-representation}$. Let $\Omega := E_+\topdual$ and let $f : \Omega \to E\topdual$ denote the inclusion. If $\Omega_d$ denotes $\Omega$ with the discrete topology, then $f$ is continuous as a map $\Omega_d \to E_\weakstar\topdual$, so we obtain a representation $f^\OEt : E \to \Cp(\Omega_d)$ via the methods of \mysecref{subsec:transpose}. The compact subsets of $\Omega_d$ are precisely the finite sets, so it follows from \myautoref{lem:representation}{itm:repr:Ck-continuous} that $f^\OEt$ is continuous as a map $E \to \Ck(\Omega_d)$. Likewise, it follows from \autoref{lem:representation} that $f^\OEt$ is positive (since $\Omega_d \subseteq E_+\topdual$) and injective (since $f[\Omega_d]$ separates points on $E$).
	
	$\myref{itm:pos:Ck-representation} \Longrightarrow \myref{itm:pos:Cp-representation}$. Choose a locally compact Hausdorff space $\Omega$ and an injective positive representation $T : E \to \Ck(\Omega)$. The topology of pointwise convergence is weaker (i.e.~coarser) than the topology of compact convergence, so $T$ is also continuous as a map $E \to \Cp(\Omega)$.
	
	$\myref{itm:pos:Cp-representation} \Longrightarrow \myref{itm:pos:coarser-normality}$. Choose a topological space $\Omega$ and an injective positive representation $T : E \to \Cp(\Omega)$. Since $T$ is injective, $E$ is linearly isomorphic to a subspace of $\Cp(\Omega)$. Let $\lintop$ denote the subspace topology that $E$ inherits this way. Since $T : E \to \Cp(\Omega)$ is continuous, $\lintop$ is weaker than the original topology of $E$. Note that $\lintop$ is the locally convex topology given by the family of seminorms $\{\seminorm_\omega\}_{\omega \in \Omega}$, where $\seminorm_\omega(x) := |T(x)(\omega)|$. Since $T$ is positive, each of these seminorms is monotone, so it follows that $E_+$ is $\lintop$-normal.
	
	$\myref{itm:pos:coarser-normality} \Longrightarrow \myref{itm:pos:weak-closure-is-proper}$. Let $\lintop$ be a weaker locally convex topology for which $E_+$ is normal. Then $(E_\lintop)\topdual \subseteq E\topdual$, since every continuous functional $E_\lintop \to \R$ is certainly continuous $E \to \R$. But $(E_\lintop)\topdual$ separates points on $E$ (since $\lintop$ is locally convex), so it follows that $E\topdual$ separates points on $E$ as well.
	
	Since $E_+$ is $\lintop$-normal, the $\lintop$-closure $\overline{E_+}^{\,\lintop}$ is a proper cone. Seeing as $E_+$ is convex and $\lintop$ is locally convex, the $\lintop$-closure of $E_+$ coincides with the respective weak closure; that is:
	\[ \overline{E_+}^{\,\lintop} \, = \, \overline{E_+}^{\,\sigma(E,(E_\lintop){}\topdual)}. \]
	Since $(E_\lintop)\topdual \subseteq E\topdual$, the weak topology $\sigma(E,E\topdual)$ is finer than the weak topology $\sigma(E,(E_\lintop)\topdual)$. It follows that
	\[ \overline{E_+}^{\,\weak} \, := \, \overline{E_+}^{\,\sigma(E,E\topdual)} \, \subseteq \, \overline{E_+}^{\,\sigma(E,(E_\lintop){}\topdual)} \, = \, \overline{E_+}^{\,\lintop}, \]
	and we conclude that $\overline{E_+}^{\,\weak}$ is a proper cone.
	
	$\myref{itm:pos:weak-closure-is-proper} \Longrightarrow \myref{itm:pos:dual-cone-separates-points}$. Let $E_+\topdualdual \subseteq E$ denote the dual cone of $E_+\topdual \subseteq E\topdual$ under the dual pair $\langle E\topdual , E\rangle$. Then $(E_+\topdual)^\perp = E_+\topdualdual \cap -E_+\topdualdual$, and by the bipolar theorem one has $E_+\topdualdual = \overline{E_+}^{\,\weak}$. Therefore $E_+\topdual$ separates points on $E$ (that is, $(E_+\topdual)^\perp = \{0\}$) if and only if $\overline{E_+}^{\,\weak}$ is a proper cone.
	
	\emph{Normed case, $\myref{itm:pos:dual-cone-separates-points} \Longrightarrow \myref{itm:pos:Ck-representation}$.} Let $\Omega := E_+\topdual \cap B_{E\topdual}$ with the relative weak\nobreakdash-$*$ topology, and let $f : \Omega \to E\topdual$ be the inclusion. Then $\Omega$ is a compact Hausdorff space and $f[\Omega]$ is equicontinuous (since $f[\Omega] \subseteq B_{E\topdual}$), so the corresponding representation $f^\OEt : E \to \Cp(\Omega)$ is continuous as a map $E \to \Ck(\Omega)$, by \myautoref{lem:representation}{itm:repr:Ck-continuous}. Likewise, it follows from \autoref{lem:representation} that $f^\OEt$ is positive (because $f[\Omega] \subseteq E_+\topdual$) and injective (since $E_+\topdual \cap B_{E\topdual}$ separates points on $E$).
	
	\emph{Normed case, $\myref{itm:pos:Ck-representation} \Longrightarrow \myref{itm:pos:coarser-normality}$.} Let $\Omega$ be a compact Hausdorff space and let $T : E \to \Ck(\Omega)$ be an injective positive representation. Note that $\Ck(\Omega)$ is simply the standard normed space $C(\Omega)$ (with the supremum norm). Since $T$ is injective and positive, the function $\lVert x \rVert_\Omega := \lVert T(x) \rVert_\infty$ defines a monotone norm on $E$. Then $E_+$ is normal with respect to $\lVert \,\cdot\, \rVert_\Omega$ (because $\lVert \,\cdot\, \rVert_\Omega$ is monotone), and the topology induced by $\lVert \,\cdot\, \rVert_\Omega$ is weaker than the original topology (because $E \to \Ck(\Omega)$ is continuous).
\end{proof}

\begin{corollary}
	\label{cor:pos:regular}
	For a convex cone $E_+$ in a real vector space $E$, the following are equivalent:
	\begin{enumerate}[label=(\roman*)]
		\item $E_+$ is regular \textup(in the sense of \autoref{thm:Schaefer-regular}\textup);
		\item\label{itm:pos:intersection-hyperplanes} The intersection of all supporting hyperplanes of $E_+$ is $\{0\}$;
		\item There is a topological space $\Omega$ and an injective positive linear map $E \to C(\Omega)$;
		\item There is a locally compact Hausdorff space $\Omega$ and an injective positive linear map $E \to C(\Omega)$.
	\end{enumerate}
\end{corollary}
\begin{proof}
	Equip $E$ with the finest locally convex topology, so that $E\topdual = E\algdual$, all subspaces are closed, and all linear maps $E \to F$ (with $F$ locally convex) are continuous. Then the result follows immediately from \autoref{thm:pos:semisimple}.%
	\hair\footnote{Recall that $C(\Omega) = \Cp(\Omega) = \Ck(\Omega)$ as vector spaces; the subscript merely indicates the choice of topology.}
\end{proof}

\begin{remark}
	\label{rmk:normed-direct-proof}
	In the proof of \autoref{thm:pos:semisimple}, the normed case was handled via the implications $\myref{itm:pos:dual-cone-separates-points} \Longrightarrow \myref{itm:pos:Ck-representation} \Longrightarrow \myref{itm:pos:coarser-normality}$. A direct proof of the implication $\myref{itm:pos:weak-closure-is-proper} \Longrightarrow \myref{itm:pos:coarser-normality}$ in the normed case will be given in \cite{order-unitizations}. There it will be shown that if $\lVert \,\cdot\, \rVert_1$ is a norm on $E$ for which $\overline{E_+}$ is a proper cone, then
	\[ \lVert x \rVert_2 := \max\!\big(d_1(x,E_+) , d_1(-x,E_+)\big) \]
	defines a smaller norm on $E$ for which $E_+$ is normal. (Here $d_1(x,E_+) = \inf_{y\in E_+} \lVert x - y \rVert_1$ denotes the $\lVert \,\cdot\, \rVert_1$-distance between $x$ and $E_+$.) Furthermore, this procedure corresponds with taking the full hull of the open unit ball; see \cite{order-unitizations}.
\end{remark}

\begin{remark}
	\label{rmk:radicals}
	We give another algebraic interpretation of semisimplicity and regularity, in terms of \emph{order ideals}, as defined by Kadison in the 1950s (\cite{Kadison-representation}; see also \cite{Bonsall}).
	
	If $E$ is an ordered vector space, then a subspace $I \subseteq E$ is an \emph{order ideal} if $I \cap E_+$ is a face of $E_+$, or equivalently, if the pushforward of $E_+$ along the quotient $E \to E/I$ is a proper cone. (For a proof of equivalence, and for additional equivalent definitions, see \cite[Proposition A.2]{ordered-tensor-products-i}.)
	
	An order ideal $I \subseteq E$ is \emph{proper} if $I \neq E$, and \emph{maximal} if it is proper and not contained in another proper order ideal.
	It follows from \cite[Theorem 2]{Bonsall} that an order ideal $I \subseteq E$ is maximal if and only if $E/I$ is one-dimensional, so the maximal order ideals are precisely the supporting hyperplanes of $E_+$ (cf.~\autoref{prop:supporting-hyperplanes}).
	
	In analogy with the Jacobson radical from abstract algebra, we define the \emph{order radical} $\orad(E_+)$ of $E_+$ as the intersection of all maximal order ideals, and the \emph{topological order radical} $\torad(E_+)$ of $E_+$ as the intersection of all \emph{closed} maximal order ideals.
	By \myautoref{cor:pos:regular}{itm:pos:intersection-hyperplanes}, $E_+$ is regular if and only if $\orad(E_+) = \{0\}$, and by \myautoref{thm:pos:semisimple}{itm:pos:intersection-closed-hyperplanes}, $E_+$ is semisimple if and only if $\torad(E_+) = \{0\}$.
	This shows that regularity and semisimplicity are closely related to \emph{Jacobson semisimplicity} in abstract algebra.
	
	We note that these order radicals have a more natural geometric definition.
	It follows from the proof of \autoref{thm:pos:semisimple}, $\myref{itm:pos:dual-cone-separates-points} \Longleftrightarrow \myref{itm:pos:intersection-closed-hyperplanes}$ and $\myref{itm:pos:weak-closure-is-proper} \Longrightarrow \myref{itm:pos:dual-cone-separates-points}$ that
	\[ \torad(E_+) \, = \, \bigcap_{\varphi\in E_+\topdual} \ker(\varphi) \, = \, (E_+\topdual)^\perp \, = \, \lineal(\overline{E_+}^{\,\weak}). \]
	Analogously, the order radical $\orad(E_+)$ is equal to the lineality space of the closure of $E_+$ in the finest locally convex topology.
	We believe that this geometric interpretation is more natural, and we have not found any real benefit from the algebraic interpretation presented here.
\end{remark}

\section{Weakly closed cones and bipositive representations}
\label{sec:bipositive}
The representations $E \to \Cp(\Omega)$ and $E \to \Ck(\Omega)$ obtained in \autoref{thm:pos:semisimple} are injective and positive, but not necessarily bipositive. We proceed to state similar criteria for the existence of bipositive injective representations $E \to \Cp(\Omega)$ and $E \to \Ck(\Omega)$, so that the positive cone $E_+$ can be written as the pullback of a cone of nonnegative continuous functions.

Schaefer \cite[Theorem 5.1]{Schaefer-I} proved that a preordered locally convex space $E$ is topologically order isomorphic with a subspace of some $\Ck(\Omega)$ ($\Omega$ locally compact and Hausdorff) if and only if $E_+$ is closed and normal (this is \autoref{thm:Schaefer-representation} in the introduction). If we remove normality, we obtain the following equivalent criteria for the existence of an injective bipositive representation $E \to \Cp(\Omega)$ or $E \to \Ck(\Omega)$.

\begin{theorem}
	\label{thm:bip:semisimple}
	For a convex cone $E_+ $ in a real topological vector space $E$, the following are equivalent:
	\begin{enumerate}[label=(\roman*)]
		\item $E_+$ is weakly closed and semisimple;
		\item\label{itm:bip:weakly-closed-proper} $E_+$ is a weakly closed proper cone;
		\item\label{itm:bip:coarser-normality} There is a weaker locally convex topology on $E$ for which $E_+$ is closed and normal;
		\item\label{itm:bip:Cp-representation} There is a topological space $\Omega$ and an injective bipositive representation $E \to \Cp(\Omega)$;
		\item\label{itm:bip:Ck-representation} There is a locally compact Hausdorff space $\Omega$ and an injective bipositive representation $E \to \Ck(\Omega)$.
	\end{enumerate}
	Furthermore, if $E$ is a normed space, then the topology in \myref{itm:bip:coarser-normality} can be taken normable, and the space $\Omega$ in \myref{itm:bip:Ck-representation} can be taken compact.
\end{theorem}
\begin{proof}
	The proof is the same as for \autoref{thm:pos:semisimple}, with the following modifications:
	
	$\myref{itm:bip:weakly-closed-proper} \Longrightarrow \myref{itm:bip:Ck-representation}$. To show that the constructed representation $E \to \Ck(\Omega)$ is bipositive, use \myautoref{lem:representation}{itm:repr:bipositive}.
	
	$\myref{itm:bip:Cp-representation} \Longrightarrow \myref{itm:bip:coarser-normality}$. If $T : E \to \Cp(\Omega)$ is bipositive and continuous, then it is also weakly continuous, and $E_+ = T^{-1}[\Cp(\Omega)_+]$. Since $\Cp(\Omega)_+$ is weakly closed, it follows that $E_+$ is weakly closed.
\end{proof}

In the normed case, a direct proof of $\myref{itm:bip:weakly-closed-proper} \Longrightarrow \myref{itm:bip:coarser-normality}$ follows from \cite{order-unitizations}, using the formula from \autoref{rmk:normed-direct-proof}.

Analogously, we get the following bipositive/weakly closed version of \autoref{cor:pos:regular}.

\begin{corollary}
	\label{cor:bip:regular}
	For a convex cone $E_+$ in a real vector space $E$, the following are equivalent:
	\begin{enumerate}[label=(\roman*)]
		\item $E_+$ is a closed proper cone with respect to the finest locally convex topology on $E$;
		\item There is a locally convex topology on $E$ for which $E_+$ is closed and regular;
		\item There is a locally convex topology on $E$ for which $E_+$ is a closed proper cone;
		\item There is a locally convex topology on $E$ for which $E_+$ is closed and normal;
		\item There is a topological space $\Omega$ and an injective bipositive linear map $E \to C(\Omega)$;
		\item There is a locally compact Hausdorff space $\Omega$ and an injective bipositive linear map $E \to C(\Omega)$.
	\end{enumerate}
\end{corollary}

\section{Hereditary properties of semisimplicity}
\label{sec:stability}
We proceed to show that semisimplicity is preserved by subspaces, products, and direct sums, but not by (proper) quotients or completions.

In addition to the results from this section, we will prove in \cite{ordered-tensor-products-i} that semisimplicity is preserved by (algebraic) tensor products.
Since it is not preserved by completions, we do not know whether semisimplicity is also preserved by \emph{completed} locally convex tensor products; see \cite[Question 5.16]{ordered-tensor-products-i}.

\subsection{Subspaces}
It is easy to see that semisimplicity is preserved by passing to a subspace, a smaller cone, and a finer topology. This follows from the following general principle.

\begin{proposition}
	\label{prop:semisimple-subspace}
	Let $E$ and $F$ be preordered topological vector spaces and let $T : E \to F$ be continuous, positive, and injective. If $F$ is semisimple, then so is $E$.
\end{proposition}
\begin{proof}
	If $S : F \to \Cp(\Omega)$ is an injective positive representation of $F$, then the composition $S\circ T : E \to F \to \Cp(\Omega)$ is an injective positive representation of $E$.
\end{proof}

\subsection{Products and direct sums}
We prove that semisimplicity is preserved by topological products, linear topological direct sums, and locally convex direct sums.

Let $\{E_\alpha\}_{\alpha \in A}$ be a family of topological vector spaces. Following Jarchow \cite[\S 4.3]{Jarchow}, we define the \emph{\textup(linear\textup) topological direct sum} $\linsum_\alpha E_\alpha$ as the vector space $\bigoplus_\alpha E_\alpha$ equipped with the finest linear topology for which the canonical injections $\iota_\beta : E_\beta \to \bigoplus_\alpha E_\alpha$ are continuous.

If all of the $E_\alpha$ are locally convex, then the \emph{locally convex direct sum} $\lcsum_\alpha E_\alpha$ is the same vector space $\bigoplus_\alpha E_\alpha$, but this time equipped with the finest \emph{locally convex} topology for which the injections $\iota_\beta : E_\beta \to \bigoplus_\alpha E_\alpha$ are continuous.
If $A$ is countable, then these two topologies on $\bigoplus_\alpha E_\alpha$ coincide (cf.~\cite[Proposition 6.6.9]{Jarchow}), but in general they can be different (cf.~\cite[Example 6.10.L]{Jarchow}).

The topological product $\prod_\alpha E_\alpha$ is simply the direct product equipped with the product topology, which is once again a topological vector space.
The relative product topology on $\bigoplus_\alpha E_\alpha$ is coarser than the (topological or locally convex) direct sum topology, so the natural map $\bigoplus_\alpha E_\alpha \to \prod_\alpha E_\alpha$ is continuous for either topology on $\bigoplus_\alpha E_\alpha$.

If each $E_\alpha$ is preordered by a convex cone $E_\alpha^+$, then we understand $\prod_\alpha E_\alpha$ (resp.~$\bigoplus_\alpha E_\alpha$) to be preordered by the pointwise cone $\prod_\alpha E_\alpha^+$ (resp.~$\bigoplus_\alpha E_\alpha^+$).

\begin{proposition}
	Let $\{E_\alpha\}_{\alpha \in A}$ be a family of preordered topological vector spaces. Then the following are equivalent:
	\begin{enumerate}[label=(\roman*),series=sums-products]
		\item\label{itm:linsum} the positive cone of the linear topological direct sum $\linsum_\alpha E_\alpha$ is semisimple;
		\item\label{itm:prod} the positive cone of the topological product $\prod_\alpha E_\alpha$ is semisimple;
		\item\label{itm:factors} for every $\alpha \in A$, the positive cone of $E_\alpha$ is semisimple.
	\end{enumerate}
	If each $E_\alpha$ is locally convex, then the preceding conditions are also equivalent to:
	\begin{enumerate}[resume*=sums-products]
		\item\label{itm:lcsum} the positive cone of the locally convex direct sum $\lcsum_\alpha E_\alpha$ is semisimple.
	\end{enumerate}
\end{proposition}
\begin{proof}
	$\myref{itm:linsum} \Longrightarrow \myref{itm:factors}$. This follows from \autoref{prop:semisimple-subspace}, since for every $\beta \in A$ the natural map $\iota_\beta : E_\beta \to \linsum_\alpha E_\alpha$ is injective, continuous, and positive.
	
	$\myref{itm:factors} \Longrightarrow \myref{itm:prod}$. We show that the continuous positive linear functionals separate points on $\prod_\alpha E_\alpha$. Suppose that $(x_\alpha)_\alpha \in \prod_\alpha E_\alpha$ is non-zero. Choose some $\beta \in A$ such that $x_\beta \neq 0$. Since $E_\beta$ is semisimple, there is a continuous positive linear functional $\varphi : E_\beta \to \R$ such that $\varphi(x_\beta) \neq 0$. The projection $\pi_\beta : \prod_\alpha E_\alpha \to E_\beta$ is continuous, linear, and positive, so it follows that $\varphi \circ \pi_\beta$ is a continuous positive linear functional on $\prod_\alpha E_\alpha$ with $(\varphi \circ \pi_\beta)((x_\alpha)_\alpha) \neq 0$.
	
	$\myref{itm:prod} \Longrightarrow \myref{itm:linsum}$. This follows from \autoref{prop:semisimple-subspace}, since the inclusion $\linsum_\alpha E_\alpha \to \prod_\alpha E_\alpha$ is continuous and positive.
	
	\emph{Locally convex case, $\myref{itm:prod} \Longrightarrow \myref{itm:lcsum} \Longrightarrow \myref{itm:linsum}$.} This follows from \autoref{prop:semisimple-subspace}, since the locally convex direct sum topology $\lcsum_\alpha E_\alpha$ is coarser than the topological direct sum topology $\linsum_\alpha E_\alpha$ and finer than the relative product topology on the subspace $\bigoplus_\alpha E_\alpha \subseteq \prod_\alpha E_\alpha$.
\end{proof}

\subsection{Quotients}
Semisimplicity is not preserved by proper quotients.

Evidently the pushforward of a semisimple cone along a quotient $E \to E/I$ is not necessarily a proper cone. (Example: if $I \subseteq E$ is a non-supporting linear hyperplane, then $E/I \cong \R$ and the pushforward of $E_+$ is all of $\R$.) But even if the quotient is proper (i.e.~if $I$ is an \emph{ideal} in the sense of \cite{Kadison-representation} and \cite{Bonsall}; see also \cite[Appendix A.1]{ordered-tensor-products-i}), it might still fail to be semisimple, as the following example shows.

\begin{example}
	Let $E = \R^3$ with the second-order cone $E_+ = \{(x_1,x_2,x_3) \, : \, \sqrt{x_1^2 + x_2^2} \leq x_3\}$. Then $E_+$ is a closed proper cone, and is therefore semisimple. However, if $x \in E_+ \setminus \{0\}$ defines an extremal ray of $E_+$, then we claim that the pushforward of $E_+$ along the quotient $E \to E/\spn(x)$ is the union of an open half-plane and $\{0\}$. Indeed, since $E_+$ is the convex cone generated by the unit disc in the $x_3 = 1$ plane, the quotient of $E_+$ by an extremal ray is generated by a disc in the plane with the origin on its boundary. Such a quotient is proper, but not semisimple (i.e.~its closure is not a proper cone).
\end{example}

If $E$ is locally convex and if $I \subseteq E$ is a closed subspace, then it follows from \cite[Proposition 2a]{positive-extensions} and \myautoref{thm:pos:semisimple}{itm:pos:dual-cone-separates-points} that the pushforward of $E_+$ along the quotient $E \to E/I$ is semisimple if and only if $\spn((I^\perp)_+)$ is weak\nobreakdash-$*$ dense in $I^\perp$, where $(I^\perp)_+ := I^\perp \cap E_+'$ denotes the positive part of $I^\perp$.
In particular, if $E$ is finite-dimensional, then the pushforward of $E_+$ along the quotient $E/I$ is semisimple if and only if the positive cone of $I^\perp$ is generating. In the preceding example, $I^\perp$ is a supporting hyperplane of $E_+\algdual \cong E_+$, but the positive part of $I^\perp$ is only one-dimensional, so not generating.

\subsection{Completions}
Semisimplicity is not preserved by completions.

\begin{example}
	\label{xmpl:completion}
	Let $E = C^1[0,1]$ with the norm $\lVert f\rVert_E = \lVert f\rVert_\infty + |f'(0)|$ and the standard cone $E_+ = \{f \in C^1[0,1] \, : \, f(x) \geq 0\ \text{for all $x\in[0,1]$}\}$. Since $E_+$ is closed with respect to the smaller (i.e.~topologically weaker/coarser) norm $\lVert \,\cdot\, \rVert_\infty$, it is also closed with respect to $\lVert \,\cdot\, \rVert_E$, so $E_+$ is semisimple in $E$.
	
	We show that $E_+$ is not semisimple in $\tilde E$. Let $g : \R \to \R$ be a continuously differentiable function with $0 \leq g(x) \leq 1$ for all $x\in \R$ and with $g'(0) \neq 0$. Define $\{g_n\}_{n=1}^\infty$ by $g_n(x) = \frac{g(nx)}{n}$. Then $g_n$ is continuously differentiable with $g_n'(x) = g'(nx)$, so we have
	\[ \lVert g_n - g_m\rVert_E = \sup_{x\in[0,1]} | g_n(x) - g_m(x) | + | g'(0) - g'(0) | \leq \frac{1}{n} + \frac{1}{m}. \]
	It follows that $\{g_n\}_{n=1}^\infty$ is Cauchy in $E$. Let $\tilde g$ denote the limit of $\{g_n\}_{n=1}^\infty$ in $\tilde E$. Since $g_n \in E_+$ for all $n \in \N_1$, it follows that $\tilde g \in \widetilde{E_+}$ (the closure of $E_+$ in $\tilde E$). 
	
	If $\one \in E$ denotes the constant function $x \mapsto 1$, then $\lim_{n\to\infty} \frac{1}{n}\one = 0$, so $\lim_{n\to\infty} g_n - \frac{1}{n}\one = \tilde g$. Since $g_n - \frac{1}{n}\one \in -E_+$, it follows that $\tilde g \in -\widetilde{E_+}$. But $\tilde g \neq 0$, because $\lVert g_n\rVert_E \geq |g_n'(0)| = |g'(0)| \neq 0$ for all $n \in \N_1$ (the sequence is bounded away from $0$), so $\widetilde{E_+}$ is not a proper cone.
\end{example}

Note: since the strong (= norm) bidual of a normed space $E$ isometrically contains the completion $\tilde E$, the preceding example also shows that the strong bidual of a semisimple ordered topological vector space is not necessarily semisimple.

\section{Some examples and counterexamples}
\label{sec:examples}
We present a few examples of semisimple cones, as well as examples of cones which are regular but not semisimple.

\subsubsection*{Finite-dimensional cones}
By \autoref{thm:pos:semisimple} and \autoref{thm:Schaefer-regular}, in a finite-dimensional space the classes of normal, regular, and semisimple cones all coincide, and this is just the class of convex cones whose closure is a proper cone. Notable cones outside this class include the lexicographic cones.

\subsubsection*{Classical function and sequence spaces}
In most classical function spaces, the natural cone of non-negative functions is closed and proper. If the topology of such a space is locally convex, then the positive cone is also weakly closed (because it is convex), and therefore semisimple.
Examples of this type include classical Banach spaces such as $C_b(\Omega)$ and $L^p(\Omega,\Sigma,\mu)$ ($1 \leq p \leq \infty$), and traditional Fr\'echet spaces such as $C^\infty(\R^n)$ and the Schwartz space $\mathcal S(\R^n)$.
Likewise, the positive cone of most classical sequence spaces is semisimple.

That these ordered topological vector spaces are semisimple is not so surprising, seeing as one of the equivalent definitions of semisimplicity is allowing a representation to a space of continuous functions. If $E$ is a space of functions $\Omega \to \R$, and if $\Omega_d$ denotes $\Omega$ equipped with the discrete topology, then the natural representation $E \to \Cp(\Omega_d)$ is continuous if all the point evaluations are. (For $L^p$ spaces one has to do a bit more work, since an $L^p$ space is actually not a space of functions, but rather a \emph{quotient} of a space of functions. In particular, the point evaluation functionals might not be well-defined.)

Things can go wrong for function spaces that are not locally convex. A notable counterexample is the space $L^p[0,1]$ for $p \in (0,1)$. Although the positive cone is still closed and proper, there are no (non-zero) continuous linear functionals, so the positive cone fails to be semisimple in a rather dramatic way. (In fact, it is well known that this cone is not even regular, since there are no positive linear functionals on this space; this follows for instance from \cite[Corollary 2.34]{Aliprantis-Tourky}.)

\subsubsection*{\texorpdfstring{$C^*$}{C*}-algebras}
The positive cone of a $C^*$-algebra is closed and proper, and therefore also weakly closed (since it is a convex set in a locally convex space).
It follows that the self-adjoint part of a $C^*$-algebra is a (real) semisimple ordered topological vector space.
(It is well-known that $C^*$-algebras are also semisimple in the algebraic sense.)

\subsubsection*{An example of a regular cone which is not semisimple}
The following example shows that semisimplicity is slightly stronger than regularity.

\begin{example}
	\label{xmpl:not-topologically-order-semisimple}
	Let $E := \R[X]$ be the algebra of real-valued polynomials in one variable, let $a,b\in\R$ be real numbers satisfying $1 \leq a < b$, and let $E_+ \subseteq E$ be the cone of polynomials that are non-negative on $[a,b]$ (i.e.~the cone inherited from the inclusion $\R[X] \hookrightarrow C[a,b]$). Then $E_+$ is Archimedean and regular.
	
	We equip $E$ with the coefficient-wise supremum norm $\lVert\:\cdot\:\rVert_{00}$ (i.e.~the norm obtained from the natural linear isomorphism $\R[X] \cong c_{00} \subseteq c_0$). We show that $E_+$ is not semisimple. In fact, we prove something stronger: $E_+$ is dense in $E$. Note that it follows from this that $E$ does not admit (non-zero) continuous positive linear functionals.
	
	To prove our claim, let $g\in E$ be any polynomial, and let $\varepsilon > 0$ be given. Since we assumed $a \geq 1$, the series $\sum_{n=0}^\infty a^n$ diverges, so we may choose $N \in \N$ such that $\sum_{n=0}^N a^n > \frac{\lVert g\rVert_{[a,b]}}{\varepsilon}$ (where $\lVert g\rVert_{[a,b]}$ denote the maximum of $|g|$ on the interval $[a,b]$). Then, by monotonicity, we have $\sum_{n=0}^N x^n > \frac{\lVert g\rVert_{[a,b]}}{\varepsilon}$ for all $x\in [a,b]$, so it follows that the polynomial $g + \varepsilon \sum_{n=0}^N x^n$ is non-negative on $[a,b]$, and therefore belongs to $E_+$. This shows that $d(g,E_+) \leq \varepsilon$, where the distance is taken with respect to the $\lVert\:\cdot\:\rVert_{00}$-norm on $E$. As this holds for every $\varepsilon > 0$, it follows that $g \in \overline{E_+}$. We conclude that $E_+$ is dense in $E$, which proves our claim.
\end{example}

\subsubsection*{Spaces of unbounded functions}
We present another class of convex cones which are regular but not semisimple.
For an arbitrary set $\Omega$, we let $\R^\Omega$ denote the vector space of all functions $\Omega \to \R$, equipped with the cone $\R_+^\Omega$ of everywhere non-negative functions.
If $\Omega_d$ denotes $\Omega$ with the discrete topology, then $\R^\Omega = C(\Omega_d)$, so it follows from \autoref{cor:pos:regular} that every subspace of $\R^\Omega$ is regularly ordered.

Let $E \subseteq \R^\Omega$ be a subspace, ordered by the induced cone $E_+ := \R_+^\Omega \cap E$.
If $E_+$ is semisimple with respect to some norm $\lVert \,\cdot\, \rVert_1$ on $E$, then it follows from \myautoref{thm:pos:semisimple}{itm:pos:coarser-normality} that $E$ admits a smaller norm $\lVert \,\cdot\, \rVert_2$ for which $E_+$ is normal, so that $\lVert \,\cdot\, \rVert_2$ is equivalent to a monotone norm $\lVert \,\cdot\, \rVert_3$ (cf.~\cite[Theorem 2.38]{Aliprantis-Tourky}). Conversely, if $\lVert \,\cdot\, \rVert_4$ is a monotone norm on $E$, then $E_+$ is automatically semisimple with respect to $\lVert \,\cdot\, \rVert_4$, by \myautoref{thm:pos:semisimple}{itm:pos:coarser-normality}.
Therefore we find that \emph{$E$ admits a norm for which $E_+$ is semisimple if and only if $E$ admits a monotone norm}.

In what follows, we present a class of function spaces which do not admit a monotone norm.
By the preceding considerations, the positive cone of such a space cannot be semisimple with respect to any norm, despite being regular.
Since every vector space $E$ admits a norm%
\hair\footnote{Proof: choose a basis $\mathcal B$ of $E$. Then $E$ is linearly isomorphic to the space $\R^{\oplus \mathcal B}$ of all functions $\mathcal B \to \R$ of finite support, which can be equipped with an $\ell^p$-norm ($1 \leq p \leq \infty$).}%
, this furnishes a class of regularly ordered normed spaces that are not semisimple.

A subset $S \subseteq \R$ is \emph{locally finite} if for every $x \in \R$ there is a neighbourhood $U_x$ of $x$ such that $S \cap U_x$ is finite, or equivalently, if $S \cap C$ is finite for every compact subset $C \subseteq \R$.%
\hair\footnote{Since $\R$ is locally compact, the second property clearly implies the first. For the converse, choose for every $x \in \R$ an open neighbourhood $U_x$ of $x$ such that $S \cap U_x$ is finite. If $C \subseteq \R$ is compact, then $C \subseteq \bigcup_{x\in \R} U_x$, so there is a finite set $F \subseteq \R$ such that $C \subseteq \bigcup_{x\in F} U_x$. Therefore $|S \cap C| \leq \sum_{x\in F} |S \cap U_x| < +\infty$.}
Furthermore, a function $g : \R \to \R$ is \emph{piecewise linear} if $g$ is continuous and there is a locally finite subset $S \subseteq \R$ such that $g$ is affine on each connected component of $\R \setminus S$.

\begin{theorem}
	\label{thm:unbounded-functions}
	\renewcommand{\labelenumi}{\roman{enumi}.}
	Let $E \subseteq \R^\Omega$ be a subspace with the following property:
	\[ \text{If $g : \R \to \R$ is piecewise linear and if $f \in E$ is arbitrary, then $g \circ f \in E$.} \tag*{$(*)$} \]
	Then $E$ admits a monotone norm if and only if every function in $E$ is bounded.
\end{theorem}
\begin{proof}
	If every function in $E$ is bounded, then $E$ can be normed with the supremum norm $\lVert \:\cdot\: \rVert_\infty$, which is monotone. For the converse, assume that $\lVert \:\cdot\: \rVert : E \to \R_{\geq 0}$ is a monotone norm. Suppose, for the sake of contradiction, that there exists an unbounded function $f \in E$. By $(*)$, we also have $|f| \in E$, so let us assume without loss of generality that $f$ is positive and unbounded. Define piecewise linear functions $g_n : \R \to \R$ by
	\[ g_n(x) \: := \: ((x - n + 1) \vee 0) \wedge 1 \: = \: \begin{cases}
		0,&\quad\text{if $x < n - 1$};\\\noalign{\smallskip}
		x - n + 1,&\quad\text{if $n - 1 \leq x \leq n$};\\\noalign{\smallskip}
		1,&\quad\text{if $x > n$}.
	\end{cases} \]
	By $(*)$, we have $g_n \circ f \in E$ for all $n\in\N_1$. Furthermore, we have $g_n \circ f \geq 0$, and $g_n \circ f \neq 0$ since $f$ is unbounded. Define $\{\alpha_n\}_{n=1}^\infty$ in $\R_{>0}$ by
	\[ \alpha_n \: := \: \frac{n}{\,\lVert g_n \circ f\rVert\,}. \]
	For every fixed $x\in\R$, at most finitely many $g_n$ take a non-zero value at $x$, so we may define $g_\infty : \R \to \R$ by
	\[ g_\infty(x) \: := \: \max_{n\in\N_1} \alpha_n g_n(x). \]
	The situation is illustrated in \autoref{fig:unbounded-functions} below.\par
	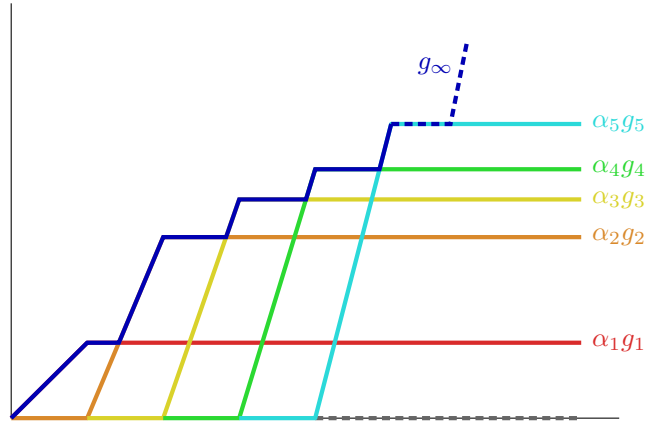
\begin{figure}[h!t]
		\centering
		\begin{tikzpicture}
			\draw (0,5.5) -- (0,0) -- (8,0);
			\draw[gray!80!black,ultra thick,densely dashed] (4,0) -- (7.5,0);
			\begin{scope}[ultra thick,line join=round]
				\definecolor{huidige}{hsb}{0,.8,1}
				\draw[huidige!85!black] (0,0) --  (1,1)  -- (7.5,1) node[right] {$\alpha_1 g_1$};
				\definecolor{huidige}{hsb}{0.09,.8,1}
				\draw[huidige!85!black] (0,0) -- (1,0) -- (2,2.4) -- (7.5,2.4) node[right] {$\alpha_2 g_2$};
				\definecolor{huidige}{hsb}{0.16,.8,1}
				\draw[huidige!85!black] (1,0) -- (2,0) -- (3,2.9) -- (7.5,2.9) node[right] {$\alpha_3 g_3$};
				\definecolor{huidige}{hsb}{0.34,.8,1}
				\draw[huidige!85!black] (2,0) -- (3,0) -- (4,3.3) -- (7.5,3.3) node[right] {$\alpha_4 g_4$};
				\definecolor{huidige}{hsb}{0.5,.8,1}
				\draw[huidige!85!black] (3,0) -- (4,0) -- (5,3.9) -- (7.5,3.9) node[right] {$\alpha_5 g_5$};
			\end{scope}
			\draw[blue!70!black,ultra thick] (0,0) -- (1,1) -- (1.41666,1) -- (2,2.4) -- (2.8276,2.4) -- (3,2.9) -- (3.8788,2.9) -- (4,3.3) -- (4.8462,3.3) -- (5,3.9);
			\draw[blue!70!black,ultra thick,densely dashed] (5,3.9) -- (5.78,3.9) to node[left,outer sep=0mm,pos=.7] {$g_\infty$} (6,5);
		\end{tikzpicture}
		\caption{$g_\infty$ is the upper envelope of the sequence $\{\alpha_n g_n\}_{n=1}^\infty$.}
		\label{fig:unbounded-functions}
	\end{figure}
	This function $g_\infty$ is again piecewise linear, so by $(*)$ we have $g_\infty \circ f \in E$. But for all $n\in\N_1$ we have $0 \leq \alpha_n g_n \circ f \leq g_\infty \circ f$, and therefore
	\[ \lVert g_\infty \circ f\rVert \: \geq \: \lVert \alpha_ng_n \circ f\rVert \: = \: n. \]
	This is a contradiction, so we conclude that every $f \in E$ must be bounded.
\end{proof}
\begin{corollary}
	Let $E \subseteq \R^\Omega$ be a space of functions which satisfies property $(*)$ from \autoref{thm:unbounded-functions} and which contains unbounded functions. Then $E_+$ is not semisimple with respect to any norm on $E$, despite being regular.
\end{corollary}

We discuss a few concrete cases. First of all, recall that a topological space $\Omega$ is called \emph{pseudocompact} if every continuous function $\Omega \to \R$ is bounded. Clearly every compact topological space is pseudocompact, but there are non-compact spaces which are nevertheless pseudocompact. (Well-known examples include the least uncountable ordinal $\omega_1$ equipped with the order topology, or $\R$ equipped with the left order topology.) Using this terminology, we have the following immediate consequence of \autoref{thm:unbounded-functions}.
\begin{corollary}
	\label{cor:pseudocompact}
	Let $\Omega$ be a topological space. Then $C(\Omega)$ admits a monotone norm if and only if $\Omega$ is pseudocompact.
\end{corollary}

For a second example, note first that \autoref{thm:unbounded-functions} does not apply to $\Ell^p$ spaces for $1 \leq p < \infty$. Although the composition $g \circ f$ is again measurable, there is no guarantee that it remains integrable. Indeed, there exist $\Ell^p$ spaces which admit a monotone norm \emph{and} contain unbounded functions. For instance, consider the measure space $(\N_1,\mathcal P(\N_1),\mu)$, where $\mu$ is the weighted counting measure given by $\mu(\{n\}) := \frac{1}{4^n}$. Now there are no non-empty null sets, so we have $\Ell^p(\mu) = L^p(\mu)$, and the $\Ell^p$ seminorm $\lVert \:\cdot\: \rVert_p$ is in fact a (monotone) norm. Nevertheless, this function space contains unbounded functions, such as $f(n) = 2^{n/p}$.

Although the theorem does not apply to $\Ell^p$ spaces in general, it does apply to all $\Ell^\infty$ spaces, for if $f$ is measurable and almost everywhere bounded, then so is $g\circ f$ for every piecewise linear function $g : \R \to \R$. As a consequence, we find that $\Ell^\infty[0,1]$ does not admit a monotone norm, since it contains unbounded functions. (Of course, it does admit a monotone \emph{seminorm}: the essential supremum seminorm.)

\section*{Acknowledgements}
I am grateful to Onno van Gaans, Marcel de Jeu, and Hent van Imhoff for many helpful comments and suggestions.
Part of this work was carried out while the author was partially supported by the Dutch Research Council (NWO), project number 613.009.127.

\phantomsection

\end{document}